\documentclass{amsart}
\usepackage[english]{babel}
\usepackage[latin1]{inputenc}
\usepackage[dvips,final]{graphics}
\usepackage{amsmath,amsfonts,amssymb,amsthm,amscd,array,stmaryrd,mathrsfs, mathdots}
\usepackage{multirow, blkarray}
\usepackage{xcolor}
\usepackage{pstricks}

 \usepackage[all]{xy}
 \usepackage{url}
\usepackage{textcomp}
 \usepackage[final]{epsfig}
\theoremstyle{plain}
\newtheorem{thm}{Theorem}%[section]
\newtheorem{lem}{Lemma}[section]

\newtheorem{prop}[lem]{Proposition}
\newtheorem{theo}[lem]{Theorem}

\theoremstyle{definition}

\newtheorem{rem}[lem]{Remark}
\newtheorem{ex}[lem]{Example}

\newcommand{\R}{\mathbb{R}}
\newcommand{\Z}{\mathbb{Z}}
\newcommand{\C}{\mathbb{C}}

\newcommand{\Q}{\mathbb{Q}}
\newcommand{\K}{\mathbb{K}}
\newcommand{\F}{\mathbb{F}}

\newcommand{\cC}{\mathcal{C}}

\newcommand{\cM}{\widehat{\mathcal{M}}}

\newcommand{\GL}{\mathrm{GL}}
\newcommand{\SL}{\mathrm{SL}}

\newcommand{\PGL}{\mathrm{PGL}}

\newcommand{\pP}{{\mathbb{P}}}

\newcommand{\id}{\textup{Id}}

\newcommand{\Char}{\textup{char}}

\newcommand{\mA}{\mathrm{A}}

\def\a{\alpha}
\def\b{\beta}

\def\l{\lambda}

\begin{document}

\title[Counting Coxeter's friezes over a finite field...]{Counting Coxeter's friezes over a finite field\\ 
via moduli spaces}

\author{Sophie Morier-Genoud}

\address{Sophie Morier-Genoud,
Sorbonne Universit\'e, Universit\'e de Paris, CNRS, 
Institut de Math\'ematiques de Jussieu-Paris Rive Gauche, IMJ-PRG, F-75005, Paris, France,
%sophie.morier-genoud@imj-prg.fr
}

\email{sophie.morier-genoud@imj-prg.fr
}

\date{}


%\subjclass{XXX}

\begin{abstract}
We count the number of Coxeter's friezes over a finite field. 
Our method uses geometric realizations of the spaces of friezes in a certain completion of the classical moduli space $\mathcal{M}_{0,n}$ allowing repeated points in the configurations. 
Counting points in the completed moduli space over a finite field is related to the enumeration problem of counting partitions of cyclically ordered set of points into subsets containing no consecutive points.  In Appendix we provide an elementary solution for this enumeration problem. \end{abstract}

\maketitle

%\tableofcontents

%\newpage

%%%%%%%%%%%%%%%%%%%%%%
%%%%%%%%%%%%%%%%%%%%%%
\section*{Introduction}
%%%%%%%%%%%%%%%%%%%%%%
%%%%%%%%%%%%%%%%%%%%%%

The notion of friezes goes back to Coxeter in the early 70's \cite{Cox}.
They are arrays of numbers in which every four adjacent values forming a square are related by the same arithmetic condition $ad-bc=1$. The classification of friezes over positive integers is given by a beautiful result of Conway and Coxeter \cite{CoCo} establishing an unexpected bijection between friezes and triangulations of polygons. As a consequence the number of friezes over positive integers with a fixed number of rows is given by a Catalan number. 

Friezes became popular in the last decade due to connections with the theory of cluster algebras of Fomin and Zelevinsky. Many variants of friezes have been recently introduced and studied. In particular Coxeter's friezes over other sets of numbers, or other types of friezes over positive integers, were investigated in various cases leading to new interesting combinatorics, e.g.
\cite{BaMa}, \cite{BeRe}, \cite{Fon}, \cite{FoPl}, \cite{Cun2}, \cite{Ovs3d},  \cite{CuHo}, \cite{HoJo3}.

In the present paper we determine the number of Coxeter's friezes over an arbitrary finite field. Our main result gives explicit polynomial formulas according to the ``width'' of the friezes, i.e. the number of rows in the friezes. We express the formulas with the help of the $q$-integer
$[m]_{q^{2}}=\frac{q^{2m}-1}{q^{2}-1}$ and the $q$-binomial coefficient 
${m \choose 2 }_{q}=\frac{(q^{m}-1)(q^{m-1}-1)}{(q-1)(q^{2}-1)}$.

\begin{thm}\label{nbfriz}(i)
The number of tame friezes of width $2m-2$ over $\F_{q}$ is given by
\begin{equation}\label{f2m}
 f_{2m-2}=[m]_{q^{2}},
  \end{equation}
 for all $m\geq 2$;
 
 (ii)
 The total number of tame friezes of width $2m-3$ over $\F_{q}$ is given by
\begin{equation}\label{f4m-1}
f_{2m-3}=
 \begin{cases}
(q-1){m \choose 2 }_{q}{}, &\text{if } \Char(\F_{q}) \not=2, \text{and } m \text { even}\\[12pt]
(q-1){m \choose 2 }_{q} + q^{m-1}, &\text{otherwise} \\[1pt]
\end{cases}
\end{equation}
for all $m\geq 2$.
\end{thm}

We solve this problem by using geometric realizations of the spaces of friezes in moduli spaces of configurations of points in the projective line. 
The realizations are different according to the parity of the number of rows in the friezes leading to different formulas according to this parity.

Friezes over complex numbers with even number of rows containing no zero entries, form an algebraic variety isomorphic to the classical moduli space $\mathcal{M}_{0,n}$, where $n$ is the width plus three, \cite{Cox}, \cite{MGOTaif}. Tame friezes are more general friezes allowing zero entries. When the number of rows is even, tame friezes form a bigger variety denoted $\cM_{0,n}$ which is some completion of $\mathcal{M}_{0,n}$ \cite{MGOST}, \cite{MGblms}. This is no longer true when the friezes have an odd number of rows, one has to consider a subvariety of $\cM_{0,n}$ that we denote $\cM^{+}_{0,n}$. 

Considering the variety of friezes as a cluster variety of type $\mA$ the result of Theorem \ref{nbfriz} agrees with the result of \cite{Cha}.

The paper is organized as follows.

Section \ref{friz} is a brief introduction to 
Coxeter's friezes. In particular we recall the definition of tame friezes and their main properties.

In Section \ref{mon} we explain the geometric realizations of the spaces of friezes.  
Cases of friezes with odd or even number of rows need to be distinguished. 
As a new result we get the geometric realization of the space of tame friezes with odd number of rows as $\cM^{+}_{0,n}$, Theorem \ref{frizmodeven}.  

In Section \ref{count} we determine the number of points in the spaces  $\cM_{0,n}$ and $\cM^{+}_{0,n}$ when considered over $\F_{q}$, Theorem \ref{contmon} and Theorem \ref{chat} respectively, and we prove Theorem \ref{nbfriz}. We end the section with examples of friezes over $\F_{q}$ and open problems.

Appendix \ref{part} can be read independently of the rest of the paper. It deals with
an enumeration problem of restricted set partitions for which we provide an elementary solution.
More precisely we establish an explicit formula for the number of partitions of $n$ points in the circle into $k$ subsets avoiding consecutive points, Proposition~\ref{Akn}. Our method is related to the counting of points in $\cM_{0,n}$ over $\F_{q}$ and requires only the formula of Lemma \ref{lemcn} established for the proof of Theorem \ref{contmon}. 
These numbers of restricted partitions are known and much information can be found about them on A261139 of OEIS.\\

%%%%%%%%%%%%%%%%%%%%%%
%%%%%%%%%%%%%%%%%%%%%%
\section{Coxeter's friezes}\label{friz}
%%%%%%%%%%%%%%%%%%%%%%
%%%%%%%%%%%%%%%%%%%%%%

\subsection{Definition}
Coxeter's friezes \cite{Cox} are arrays of numbers that can be described as follows:

(i) The array has finitely many rows, all of them being infinite on the right and left;

(ii) The first and last rows 
are rows of 1's;

(iii) Consecutive rows are displayed with a shift, and every four adjacent entries $a, b,c,d$  forming a diamond $$
 \begin{array}{ccccccc}
 &b&\\
 a&&d\\
 &c&
\end{array}$$ satisfy the \textit{unimodular rule}: $ad-bc=1$.

The number of rows strictly between the border rows of 1's is called the \textit{width} of the frieze.
The following array \eqref{exCoCo} is an example of a frieze pattern of width  $w=4$. 
\begin{equation}\label{exCoCo}
 \begin{array}{lcccccccccccccccccccccccc}
\text{\small{row 0}}&&&1&&1&& 1&&1&&1&&1&&1&& \cdots\\[4pt]
\text{\small{row 1}}&&\cdots&&4&&2&&1&&3&&2&&2&&1&&
 \\[4pt]
\text{\small{row 2}}&&&3&&7&&1&&2&&5&&3&&1&&\cdots&\
 \\[4pt]
\cdots &&\cdots& &5&&3&&1&&3&&7&&1&&2&\\[4pt]
\text{\small{row $w$}}&&& 3&&2&&2&&1&&4&&2&&1&&\cdots\\[4pt]
\text{\small{row $w+1$}}&&\cdots&&1&&1&&1&&1&&1&&1&&1&&
\end{array}
\end{equation}

One can observe that the above example of frieze contains only positive integers. However, the definition allows the frieze to take its values in any subset of a unital ring.
Conway and Coxeter \cite{CoCo} established a surprising one-to-one correspondence between the friezes over positive integers and the triangulations of polygons. As a consequence of this correspondence one gets the following.

\begin{theo}[\cite{CoCo}]
The number of friezes of width $w$ over positive integers is given by the Catalan number $C_{w+1}=\frac{1}{w+2} {{2w+2}\choose{w+1}}$.
\end{theo}

Several enumerative interpretations for the entries appearing in the friezes are known, see \cite{MGblms} for an overview.

\subsection{Tame friezes} \label{tamsec}
By convention one extends the friezes top and bottom by rows of 0's and $-1$'s so that the unimodular rule still holds.
Following \cite{BeRe} a frieze is said to be \textit{tame} if it satisfies the following extra condition: every $3\times 3$-entries in diamond in the extended array

\begin{equation}\label{tam}
 \begin{array}{ccccccccccc}
 &&c\\
 &b&&f\\
 a&&e&&i\\
 &d&&h\\
 &&g
\end{array}
\end{equation}
form a matrix of rank 2. 

Friezes with no zero entries between the bording rows of 1's are all tame, but tame friezes may contain zeroes.
Generically when $e\not=0$ the determinant of the matrix \eqref{tam} formed by the ajacent entries in a frieze vanishes by Desnanot-Jacobi identity.  Hence the tameness condition has to be checked only for entries centered at the zeroes of the frieze.

Many properties established by Coxeter in the case of friezes over positive real numbers can be directly generalized to tame friezes with entries in any unital ring. 
Non-tame friezes are called ``wild'' and may have completely different behavior, see \cite{Cun2}. 
In this paper we always assume that the friezes are tame.

Let us recollect some of the main properties of tame friezes.

\begin{theo}[\cite{Cox}]\label{thmper}
Tame friezes of width $w$ are $(w+3)$-periodic. Moreover they are invariant under a glide reflection.
\end{theo}

The above property can be observed in the frieze \eqref{exCoCo}. This array consists in the fundamental domain
$$
 \begin{array}{lcccccccccccccccccccccccc}
1&&1&& 1&&1&&1&&1&&\\
&4&&2&&1&&3&&2&&&&
 \\
&&7&&1&&2&&5&&&
 \\
 &&&3&&1&&3&&\\
&&&&2&&1&&&\\
&&&&&1&&
\end{array}
$$
repeated horizontally under glide reflections. We will sometimes represent a fundamental domain instead of the whole frieze.

The entries in the first row of a tame frieze of width $w$ will be denoted $a_{1}, \ldots, a_{n}$, with $n=w+3$.

\begin{theo}[\cite{Cox}]
The South-East diagonal $(\Delta_{i})_{0\leq i \leq w+1}$ in a tame friezes of width $w$, starting at $\Delta_{0}=1, \Delta_{1}=a_{1}$,
satisfies the recursion
\begin{equation}\label{recfri}
\Delta_{i+1}=a_{i+1}\Delta_{i}-\Delta_{i-1}, \; 1\leq i \leq w,
\end{equation}
and similarly for the other diagonals up to shifts in the indices.
\end{theo}

The above recursion gives an alternative way to compute the entries in the frieze from the first row without using the relation $ad-bc=1$. This relation may cause trouble in the case where one of the entries is zero.
A tame frieze starts as follows.
$$
 \begin{array}{lcccccccccccccccccccccccc}
1&&1&& 1&&1&&1& \\[4pt]
&a_{1}&&a_{2}&&\cdots&&a_{n}&&a_{1}&&
 \\[4pt]
\cdots&&a_{1}a_{2}-1&&a_{2}a_{3}-1&&\cdots&&a_{n}a_{1}-1&&
 \\[5pt]
&a_{n}a_{1}a_{2}-a_{2}&&a_{1}a_{2}a_{3}-a_{3}&&\cdots&&a_{n}a_{1}a_{2}-a_{2}&&\cdots&\\[-2pt]
&-a_{n}&&-a_{1}&&&&-a_{n}&&&&\\[4pt]
\ddots&&\ddots&&\ddots&&\ddots&&\ddots\\[6pt]
&1&&1&&1&&1&&1
\end{array}
$$
The recursion \eqref{recfri} remains valid in the frieze extended top and bottom with rows of 0's and $-1$'s. This allows to immediately deduce the glide reflection property of Theorem \ref{thmper}.
Furthermore, the recursion \eqref{recfri} in the extended frieze can be encoded by $2\times2$ matrices and lead to the following criterion.
\begin{theo}[\cite{Cox}]\label{crit}
The $n$-tuple $(a_{1}, \ldots, a_{n})\in \K^{n}$ determines the first row of a tame frieze if and only if
\begin{equation}\label{mateq}
\left(
\begin{array}{cc}
a_n&-1\\[4pt]
1&0
\end{array}
\right)
\cdots
\left(
\begin{array}{cc}
a_{2}&-1\\[4pt]
1&0
\end{array}
\right)
\left(
\begin{array}{cc}
a_{1}&-1\\[4pt]
1&0
\end{array}
\right)
=-\id.
\end{equation}
\end{theo}
The condition \eqref{mateq} is established in \cite{Cox} in an equivalent form, using continuants, in the case of friezes over positive real numbers. It was also recovered in this case in \cite{BeRe}. The arguments can be extended straightforward in the case of tame friezes over any unital ring. The condition \eqref{mateq} is also established in \cite{CuHo} for tame friezes over subsets of $\C$ and again the arguments can be extended straightforward in the case of tame friezes over arbitrary unital ring.

%%%%%%%%%%%%%%%%%%%%%%
%%%%%%%%%%%%%%%%%%%%%%
\section{Friezes and the moduli space $\cM_{0,n}$}\label{mon}
%%%%%%%%%%%%%%%%%%%%%%
%%%%%%%%%%%%%%%%%%%%%%

\subsection{Definition of $\cM_{0,n}$}\label{monq}
We fix a ground field $\K$. The field will be later specified to the field of complex numbers $\C$ or to the finite field $\F_{q}$ with $q$ elements.
A variety $V(\K)$ defined over $\K$ will be denoted simply $V$ when no confusion occurs, e.g 
$\pP^{1}=\pP^{1}(\K)$, $\GL_{2}=\GL_{2}(\K)$.
The classical moduli space of genus 0 curves with $n$ distinct marked points is
$$
\mathcal{M}_{0,n}\simeq\{ (v_{1}, v_{2}, \ldots, v_{n})\in (\pP^{1})^{n}, \;v_{i}\not=v_{j},\,\forall i\not=j\}/\PGL_{2}.
$$
We consider a completion of this space by allowing repetitions of points as long as they are not cyclically next to each other
$$
\cM_{0,n}=\{ (v_{1}, v_{2}, \ldots, v_{n})\in (\pP^{1})^{n}, \;v_{i}\not=v_{i+1}, \,\forall i\}/\PGL_{2},
$$
The indices are considered modulo $n$, i.e. $v_{1}=v_{n+1}\not=v_{n}$.

We often consider the configuration space used in the definition of $
\cM_{0,n}$; we define
$$
\cC_{n}=\{ (v_{1}, v_{2}, \ldots, v_{n})\in (\pP^{1})^{n}, \;v_{i}\not=v_{i+1}, \, \forall i\}
$$
where again the indices are considered modulo $n$.

\subsection{The subvariety of $\cM^{+}_{0,2m}$}

The varieties $\cM_{0,n}$ have different properties according to the parity of $n$. 
In the case $n=2m$ we will need to consider the algebraic subvariety of $\cM_{0,2m}^{+}$ that we define below. 
We first introduce two subspaces of $\cC_{2m}$
$$
\cC_{2m}^{\pm}:=\left\{(v_{1},\ldots, v_{2m})\in \cC_{2m} \Big| \frac{(v_{1}-v_{2})(v_{3}-v_{4})\cdots (v_{2m-1}-v_{2m})}{(v_{2}-v_{3})(v_{4}-v_{5})\cdots (v_{2m}-v_{1})}=-(\pm 1)\right\}
$$
and then define
$$\cM_{0,2m}^{\pm}:=\cC_{2m}^{\pm}/\PGL_{2}.$$
These are algebraic subvarieties of $\cM_{0,2m}$ of codimension 1.

\begin{rem}
These subvarieties are studied in \cite{Boa} as $\mathcal{M}^{\text{Sibuya}}_{2m}(\pm1)$, see also \cite{Sib}.
\end{rem}

\subsection{Geometric realizations of tame friezes}

Coxeter's friezes have deep connections with geometry \cite{Cox}, \cite{Cox3}. 
The link with the moduli space $\mathcal{M}_{0,n}$ was explained in \cite{MGOTaif} and then further generalized in higher dimension in \cite{MGOST}. 

The realization of friezes as geometric objects is the key ingredient in order to achieve the goal of counting friezes over a finite field.

In the case of odd $n$ one has the following known result.

\begin{theo}[{\cite{MGOST}}]\label{frizmod} When $n\geq 3$ is odd,
tame friezes of width $n-3$ over $\K$ form an algebraic subvariety of $\K^{n}$ isomorphic to $\cM_{0,n}(\K)$.
\end{theo}

The above theorem is the particular case ($k=1$) of \cite[Thm 3.4.1]{MGOST}.
The set of tame friezes of width $n-3$ is viewed as the subvariety of $\K^{n}$ of dimension $n-3$ defined by the polynomial equations of \eqref{mateq} (note that the matrices involved in \eqref{mateq} belong to $\SL_{2}$ so that one gets three independent equations). The space $\cM_{0,n}(\K)$ is identified with a quotient of a Grassmannian by a torus action using Gelfand-MacPherson's correspondence~\cite{GeMc}.

For our purpose, we only need a one-to-one correspondence between the set of tame friezes and the moduli space (we do not need the isomorphism of algebraic varieties). For completeness we explain the one-to-one correspondence in Section \ref{pftheoodd}.\\

In the case of even $n$ the situation is different. We need to consider friezes up to rescaling of the first row. More precisely $\K^{*}$ acts on the space of tame friezes with odd width by sending the frieze defined by the cycle of the first row
$(a_{1}, \ldots, a_{n})$ to the frieze defined by $(\l a_{1}, \frac{1}{\l} a_{2},  \ldots, \l a_{n-1} , \frac{1}{\l}a_{n})$ for $\l\in\K^{*}$.  Also, some elements in $\cM_{0,n}$ do not correspond to friezes. We need to restrict to  $\cM^{+}_{0,n}$. One gets the following.

\begin{thm}\label{frizmodeven} When $n\geq 3$ is even, the set of
tame friezes of width $n-3$ over $\K$, modulo the action of $\K^{*}$, is in one-to-one correspondence with $\cM^{+}_{0,n}(\K)$.
\end{thm}

We prove the above theorem in Section \ref{pftheoeven}.

\subsection{A one-to-one correspondence between tame friezes and $\cM_{0,n}(\K)$, case of odd $n$}\label{pftheoodd}
We explain the bijection of Theorem \ref{frizmod} between the tame friezes of width $w=n-3$ with odd $n$ over $\K$ and the set 
$\cM_{0,n}(\K)=\cC_{n}(\K)/\PGL_{2}(\K)$. The construction was given for $\K=\R$ in \cite{MGOST} and $\K=\C$ in \cite{MGblms}. Here the construction is slightly modified in order to work over an arbitrary field $\K$.

Choose $v=(v_{1}, \ldots, v_{n})$ an element of $\cC_{n}$. 
\begin{lem}\label{lemlift}
There exists $(V_{1}, \ldots, V_{n})$ a lift of $v$ with $V_{i}\in \K^{2}$ such that 
$$
\det(V_{1}, V_{2})=\det(V_{2},V_{3})=\cdots=\det(V_{n-1},V_{n})=\det(V_{n},-V_{1}).
$$
\end{lem}
\begin{proof}
Consider first an arbitrary lift $(V'_{1}, \ldots, V'_{n})$ of $v$ in $\K^{2}$. We want to extend the sequence by antiperiodicity, $V'_{i+n}=-V'_{i}$, and rescale the points in order to have constant consecutive determinants. Introducing the scalar parameter $c$ and the unknown scalars $\lambda_{i}$, $1\leq i \leq n$, we consider the system given by the $n$ conditions $\det(\lambda_{i}V'_{i}, \lambda_{i+1}V'_{i+1})=c$, $1\leq i \leq n$. 

This leads to the following 
\begin{equation}\label{sys}
\left\{
\begin{array}{ccc}
\lambda_{1}\lambda_{2}  &=   &d_{1}c   \\
\lambda_{2}\lambda_{3}  &=   &d_{2}c   \\
\vdots  &   & \vdots  \\
\lambda_{n-1}\lambda_{n}  &=   &d_{n-1}c   \\
\lambda_{n}\lambda_{1}  &=   &d_{n}c   
\end{array}
\right.
\end{equation}
where $d_{i}=1/\det(V'_{i},V'_{i+1})$ are well defined non zero scalars since $v_{i}\not=v_{i+1}$ in the initial configuration.

Using $\lambda_{1}$ as a parameter to solve the system one gets the expressions
\begin{equation}\label{syssol}
\l_{2i}=\frac{d_{1}d_{3}\cdots d_{2i-1}}{d_{2}d_{4}\cdots d_{2i-2}}\frac{c}{\l_{1}}, \quad 
\l_{2i+1}=\frac{d_{2}d_{4}\cdots d_{2i}}{d_{1}d_{3}\cdots d_{2i-1}}\l_{1}
\end{equation}
for all $1\leq i \leq (n-1)/2$, using all but the last equation of \eqref{sys}.

 The fact that $n$ is odd is now crucial to solve the above system. 
In order to have the full system satisfied one needs to add the last equation and gets the final relation 
$$
\lambda_{1}^{2}=\frac{d_{1}d_{3}\ldots d_{n}}{d_{2}d_{4}\ldots d_{n-1}}c.
$$
It is clear that there is a choice of $c\in \K^{*}$ leading to a solution $(\lambda_{1}, \ldots, \lambda_{n})$ for the system. 
The sequence defined by $V_{i}=\lambda_{i}V'_{i}$ has the desired property.
\end{proof}
From a lift $(V_{i})$ of $v$ given by Lemma \ref{lemlift}, by expanding $V_{i}$ in the basis $(V_{i-1}, V_{i-2})$, one gets $n$ coefficients $(a_{1}, \ldots, a_{n})$ defined by 
\begin{equation}\label{eqcoeff}
V_{i}=a_{i}V_{i-1}-V_{i-2}, \; 1\leq i \leq n
\end{equation}
where we set $V_{0}=-V_{n}$ and $V_{-1}=-V_{n-1}$. 
The coefficients in front of $V_{i-2}$ are all equal to $-1$ since consecutive determinants are constant.
\begin{lem}\label{lemuniqcoef}
The coefficients $(a_{1}, \ldots, a_{n})$ defined by \eqref{eqcoeff} only depend on the class of $v$ modulo $\PGL_{2}(\K)$ and characterize the class.
\end{lem}
\begin{proof}
We need to check that the coefficients do not depend on the choice of the lift in Lemma \ref{lemlift} and that they are invariant under the action of $\PGL_{2}(\K)$.
Let $w$ be a configuration of $\cC_{n}$ equivalent to $v$ modulo $\PGL_{2}(\K)$ (this of course includes the case $w=v$). Consider lifts 
$(W_{1}, \ldots, W_{n})$ and $(V_{1}, \ldots, V_{n})$ for $w$ and $v$ as in Lemma \ref{lemlift}. 
There exist a matrix $M\in \GL_{2}(\K)$ and scalars $\mu_{1}, \ldots, \mu_{n}$ in $\K^{*}$ such that 
$W_{i}=\mu_{i}MV_{i}$ for all $1\leq i \leq n$. Since the consecutive determinants are constant in both sequences one immediately gets that $\mu_{i}\mu_{i+1}=\mu_{i-1}\mu_{i}$ for all $1\leq i \leq n$, and hence  $\mu_{i}=\mu_{1}$  for all $1\leq i \leq n$ (here again we use the fact that $n$ is odd). Thus $W_{i}=\mu_{1}MV_{i}$ for all $1\leq i \leq n$ and this leads to the same coefficients in the expansions \eqref{eqcoeff}.

Now, if two lifts $(W_{1}, \ldots, W_{n})$ and $(V_{1}, \ldots, V_{n})$ for two configurations $w$ and $v$ lead to the same coefficients \eqref{eqcoeff} then the matrix $M$ defined by $MV_{1}=W_{1}$ and $MV_{2}=W_{2}$ will satisfy $MV_{i}=W_{i}$ for all $1\leq i \leq n$. Thus $w$ and $v$ are equivalent modulo $\PGL_{2}(\K)$.
\end{proof}
\begin{lem}\label{lemquid}
The coefficients $(a_{1}, \ldots, a_{n})$ satisfy \eqref{mateq}.
\end{lem}
\begin{proof}
By writing the coordinates of the vectors $V_{i}$ and $V_{i-1}$ in the rows of a $2\times2$ matrix one obtains from \eqref{eqcoeff}
$$
\begin{pmatrix}
-V_{0}\\[4pt]-V_{-1}
\end{pmatrix}=
\begin{pmatrix}
V_{n}\\[4pt]V_{n-1}
\end{pmatrix}=
\left(
\begin{array}{cc}
a_n&-1\\[4pt]
1&0
\end{array}
\right)
\cdots
\left(
\begin{array}{cc}
a_{2}&-1\\[4pt]
1&0
\end{array}
\right)
\left(
\begin{array}{cc}
a_{1}&-1\\[4pt]
1&0
\end{array}
\right)\begin{pmatrix}
V_{0}\\[4pt]V_{-1}
\end{pmatrix}.
$$
Hence \eqref{mateq} holds.
\end{proof}
The above lemma and Theorem \ref{crit} imply that the coefficients $(a_{1}, \ldots, a_{n})$ defined by \eqref{eqcoeff} determine a unique tame frieze of width $n-3$. Hence we have constructed an injective map from $\cM_{0,n}(\K)$ to the set of tame friezes of width $n-3$. 

Conversely, starting form a tame frieze of width $n-3$, the corresponding point in $\cM_{0,n}(\K)$ is obtained from the sequence of $n$ vectors
\begin{equation}\label{seqfri}
(V_{1}, \ldots, V_{n})=\begin{pmatrix}
1\\[4pt]a_{1}
\end{pmatrix}, \begin{pmatrix}
a_{2}\\[4pt]a_{1}a_{2}-1
\end{pmatrix}, 
\begin{pmatrix}
a_{2}a_{3}-1\\[4pt]*
\end{pmatrix}, 
\cdots,
\begin{pmatrix}
*\\[4pt]1
\end{pmatrix}, 
\begin{pmatrix}
1\\[4pt]0
\end{pmatrix},
\begin{pmatrix}
0\\[4pt]-1
\end{pmatrix}
\end{equation}
whose components are given by the first and second diagonal of the frieze.

\subsection{Correspondence between tame friezes and $\cM^{+}_{0,n}(\K)$, case of even $n$}\label{pftheoeven}
In this section we prove Theorem \ref{frizmodeven}.
When $n$ is even there is no bijection between the tame friezes of width $w=n-3$ and the set 
$\cM_{0,n}(\K)$. We need to adapt the strategy used in the case of odd $n$ and use the subspace $\cM_{0,n}^{+}(\K)$.

\begin{lem}
When $n$ is even Lemma \ref{lemlift} holds if and only if $v\in \cC^{+}_{n}$.
\end{lem}
\begin{proof}
We fix $v\in \cC^{+}_{n}$ and proceed as in the proof of  Lemma \ref{lemlift}. 
We are led to the system \eqref{sys}. In general, for even $n$ the system has no solution.
But with the choice of $v\in \cC^{+}_{n}$, by definition, every arbitrary lift of $v$ to $(V'_{i})_{1\leq i \leq n}$ will satisfy the relation
$$d_{1}d_{3}\cdots d_{n-1}=d_{2}d_{4}\cdots d_{n}$$
where $d_{i}=1/\det(V'_{i},V'_{i+1})$ (with $V'_{n+1}=-V'_{1}$). This makes the system \eqref{sys} consistent and solutions for $(\lambda_{1}, \ldots, \lambda_{n})$ are given by \eqref{syssol} where $\l_{1}$ is any non zero parameter. Hence the rescaling 
$V_{i}=\lambda_{i}V'_{i}$ provides us with a antiperiodic lift with constant consecutive determinants.

Conversely, if such a lift exists the element $v$ does belong to $\cC^{+}_{n}$.
(Note that the choice of $V_{n+1}=-V_{1}$ makes a difference of sign in the equation defining $\cC^{+}_{n}$.)
\end{proof}

For $v\in \cC^{+}_{n}$ and a lift $(V_{i})$ given by Lemma \ref{lemlift} one gets $n$ coefficients $(a_{1}, \ldots, a_{n})$ defined by the recursion \eqref{eqcoeff}. Now these coefficients depend on the choice of the lift, but another choice of lift will lead to the coefficients $(\l a_{1}, \frac{1}{\l} a_{2},  \ldots, \l a_{n-1} , \frac{1}{\l}a_{n})$ for some $\l\not=0$. 
We immediately obtain the following analogs of Lemma \ref{lemuniqcoef} and Lemma \ref{lemquid}.
\begin{lem}\label{lemuniqcoef2}
The family of coefficients $(\l a_{1}, \frac{1}{\l} a_{2},  \ldots, \l a_{n-1} , \frac{1}{\l}a_{n})$, $\l\not=0$, defined by \eqref{eqcoeff}, only depends on the class of $v$ modulo the action of $\PGL_{2}(\K)$ and characterize the class.
\end{lem}
\begin{lem}\label{lemquid2}
The coefficients $(\l a_{1}, \frac{1}{\l} a_{2},  \ldots, \l a_{n-1} , \frac{1}{\l}a_{n})$ satisfy \eqref{mateq}.
\end{lem}
We conclude that  an element of $\cM^{+}_{0,n}=\cC^{+}_{n}/\PGL_{2}$ defines a unique tame frieze  up to the rescaling 
$(\l a_{1}, \frac{1}{\l} a_{2},  \ldots, \l a_{n-1} , \frac{1}{\l}a_{n})$, $\l\not=0$, in the first row, and \textit{vice versa} by \eqref{seqfri}.

%%%%%%%%%%%%%%%%%%%%%%
%%%%%%%%%%%%%%%%%%%%%%
\section{Counting points over $\F_{q}$}\label{count}
%%%%%%%%%%%%%%%%%%%%%%
%%%%%%%%%%%%%%%%%%%%%%

\subsection{Counting points  in  $\cM_{0,n}(\F_{q})$}\label{count1}
Let us recall some useful notation and basic facts. One has $[n]_{q}:=|\pP^{n-1}(\F_{q})|=1+q+q^{2}+\ldots+q^{n-1}=\frac{q^{n}-1}{q-1}$ and $|\PGL_{2}(\F_{q})|=q(q^{2}-1)$.

\begin{thm}\label{contmon} Let $n\geq 2$. The number of points in  $\cM_{0,n}(\F_{q})$ is
$$
| \cM_{0,n}(\F_{q})|=\begin{cases}
[m]_{q^{2}}, &\text{if } n =2m+1,\\[4pt]
1+q[m-1]_{q^{2}}, &\text{if } n=2m.
\end{cases}
$$
\end{thm}

\begin{proof}
We first count points in the space $
\cC_{n}$.
\begin{lem}\label{lemcn}
Denote by $c_{n}=|\cC_{n}(\F_{q})|$. One has
\begin{equation}\label{formcn}
c_{n}=q^{n}+(-1)^{n}q, \;n\geq 2.
\end{equation}
\end{lem}
\begin{proof}
One easily computes the first values $c_{2}=(q+1)q$, $c_{3}=(q+1)q(q-1)$, which are obtained as the choice of two, resp. three, distinct points among the $q+1$ points of $\pP^{1}(\F_{q})$.
Consider a configuration $(v_{1}, \ldots, v_{n+1},v_{n+2})$ in $\cC_{n+2}$. If $v_{1}=v_{n+1}$ then the configuration can be viewed as a configuration $(v_{1}, \ldots, v_{n})\in \cC_{n}$ with an extra point $v_{n+2}\not=v_{1}$. There are $qc_{n}$ configurations of this type. If $v_{1}\not=v_{n+1}$ then the configuration can be viewed as a configuration $(v_{1}, \ldots, v_{n+1})\in \cC_{n+1}$ with an extra point $v_{n+2}\not\in\{v_{1},v_{n+1}\}$. There are $(q-1)c_{n+1}$ configurations of this type. 
This establishes the recurrence relation
\begin{equation*}
c_{n+2}=(q-1)c_{n+1}+qc_{n}, \;n\geq 2,
\end{equation*}
This linear recurrence of order 2 can be easily solved and leads to the formula \eqref{formcn}.
\end{proof}
When a configuration contains three distinct points it will lead to $q(q^{2}-1)$ other configurations modulo the action of $\PGL_{2}(\F_{q})$. When $n=2m+1$ the configurations in $\cC_{n}$ always contain three distinct points, therefore
$$
| \cM_{0,2m+1}(\F_{q})|=c_{2m+1}/(q(q^{2}-1))=(q^{2m}-1)/(q^{2}-1)=[m]_{q^{2}}.
$$
When $n=2m$, the $(q+1)q$ configurations of the form $(v_{1},v_{2}, v_{1},v_{2}, \ldots, v_{1}, v_{2})$, $v_{1}\not=v_{2}$ are all the same modulo $\PGL_{2}(\F_{q})$. All other configurations contain three distinct points. Therefore
$$
| \cM_{0,2m}(\F_{q})|=(c_{2m}-q(q+1))/(q(q^{2}-1))+1=1+q[m-1]_{q^{2}}.
$$

\end{proof}
Lemma \ref{lemcn} is the only result needed in Appendix \ref{part}.

\subsection{Counting points in $\cM^{+}_{0,2m}(\F_{q})$}
We consider $\K=\F_{q}$ and want to count the points in  the variety $\cM^{+}_{0,2m}$ defined in the previous section. Different cases appear according to the characteristic of the field is 2 or not and according to $m$ is even or odd. Recall the definition of the $q$-binomial coefficient 
${m \choose 2 }_{q}=\frac{(q^{m}-1)(q^{m-1}-1)}{(q-1)(q^{2}-1)}$, which symplify to a polynomial in $q$.

\begin{thm}\label{chat}
For all $m\geq 1$
$$
|\cM^{+}_{0,2m}(\F_{q})|=
\begin{cases}
{m \choose 2 }_{q}, &\text{if } \Char(\F_{q}) \not=2 \; \text{and } m \text{ even},\\[12pt]
{m \choose 2 }_{q}+[m-1]_{q}+1, &\text{otherwise.}\\[6pt]
\end{cases}
$$
\end{thm}
\begin{proof}
We introduce the cardinals
$$
c_{n}^{\pm}=|\cC_{n}^{\pm}(\F_{q})|.
$$
Recall that $c_{n}=|\cC_{n}(\F_{q})|$ is given by Equation \eqref{formcn}.
\begin{lem}\label{lemcpm}
For all $n\geq 4$, even, one has
$
c_{n}^{\pm}=c_{n-1}+qc^{\mp}_{n-2}.
$
\end{lem}
\begin{proof}
Let us choose $v=(v_{1},\ldots, v_{n})\in \cC^{+}_{n}$. Assume first that $v_{1}\not=v_{n-1}$. 
This means that $v'=(v_{1},\ldots, v_{n-1})$ defines an element of $\cC_{n-1}$. For such an element $v'$ there is a unique choice of $v_{n}$ that makes $(v',v_{n})$ belong to $\cC^{+}_{n}$. Indeed, 
choose any lift of $v'$ to a tuple $(V_{1}, V_{2}, \ldots, V_{n-1})$ of points in $\K^{2}$ and denote $d_{i}=\det(V_{i},V_{i+1})$ for $i=1\ldots n-2$. We want to add $V_{n}$ such that the following condition 
$$d_{1}d_{3}\cdots d_{n-3}\det(V_{n-1}, V_{n})=-d_{2}d_{4}\cdots d_{n-2}\det(V_{n},V_{1})$$
holds. Since $v_{1}\not=v_{n-1}$, this
 leads to a unique choice of $V_{n}$, up to scalar multiple. 
 We conclude that there exist $c_{n-1}$ elements  $v$ in $\cC^{+}_{n}$ such that $v_{1}\not=v_{n-1}$.
 
 Assume now $v_{1}=v_{n-1}$. The sequence $v''=(v_{1},\ldots, v_{n-2})$  is an  element in $\cC_{n-2}$. Choose any lift of $v''$ to $(V_{1}, V_{2}, \ldots, V_{n-2})$ and denote $d_{i}=\det(V_{i},V_{i+1})$ for $i=1\ldots n-3$ and $d_{n-2}=\det(V_{n-2}, V_{1})$. 
 We want to add $V_{n-1}=\lambda V_{1}$ and $V_{n}$ in order to have
 $$d_{1}d_{3}\cdots d_{n-3}\det(V_{n-1}, V_{n})=-d_{2}d_{4}\cdots d_{n-3}\det(V_{n-2}, V_{n-1})\det(V_{n},V_{1}).$$
 Substituting $V_{n-1}=\lambda V_{1}$ and cancelling $\l$'s we obtain the simplified condition
 $$d_{1}d_{3}\cdots d_{n-3}\det(V_{1}, V_{n})=-d_{2}d_{4}\cdots d_{n-2}\det(V_{n},V_{1}),$$
which can hold if and only if $d_{1}d_{3}\cdots d_{n-3}=d_{2}d_{4}\cdots d_{n-2}$, i.e. if and only if $v''\in \cC_{n-2}^{-}$. And from this condition we see that  if $v''\in \cC_{n-2}^{-}$ any choice of $v_{n}$ in $\pP^{1}(\F_{q})\setminus\{v_{1}\}$ will make $(v'',v_{1}, v_{n})$ belong to~$\cC^{+}_{n}$.
 We conclude that there exist $qc_{n-2}^{-}$ elements  $v$ in $\cC^{+}_{n}$ such that $v_{1}=v_{n-1}$.

We have established the recursion $
c_{n}^{+}=c_{n-1}+qc^{-}_{n-2}$. Same arguments hold when taking opposite signs so we also have 
$c_{n}^{-}=c_{n-1}+qc^{+}_{n-2}$.
\end{proof}
From the above lemma one obtains
$$
c^{+}_{2m}=c_{2m-1}+qc_{2m-3}+q^{2}c_{2m-5}+\ldots+q^{m-2}c_{3}+q^{m-1}c_{2}^{(-)^{m-1}}
$$
where $(-)^{m-1}=+$ or $-$ if $m$ is odd or even respectively. So we will need to separate the cases 
$m$ odd or even. Also if $\Char(\F_{q})=2$ one has $c^{+}_{n}=c^{-}_{n}$ since $\cC^{+}_{n}=\cC^{-}_{n}$. So we will need to separate the cases 
$\Char(\F_{q})=2$  and $\Char(\F_{q})>2$.

\begin{lem}\label{lemchat}
For all $m\geq 1$
$$
|\cM^{+}_{0,2m}(\F_{q})|=
\begin{cases}
\sum_{k=1}^{m-1}q^{k-1}[m-k]_{q^{2}}, &\text{if } \Char(\F_{q}) \not=2 \; \text{and } m \text{ even},\\[8pt]
\sum_{k=1}^{m-1}q^{k-1}[m-k]_{q^{2}}+[m-1]_{q}+1, &\text{otherwise.}\\[6pt]
\end{cases}
$$
\end{lem}
\begin{proof}
Assume $\Char(\F_{q})>2$ and $n=2m$ with even $m$.
We have $\cC_{2}^{-}=\emptyset$ so that $c^{-}_{2}=0$ and we obtain from the recursion
$$
c^{+}_{2m}=c_{2m-1}+qc_{2m-3}+q^{2}c_{2m-5}+\ldots+q^{m-2}c_{3}.
$$
The space $\cC^{+}_{2m}$ does not contain the elements of the form $(v_{1},v_{2}, v_{1},v_{2}, \ldots, v_{1},v_{2})$. Thus elements of $\cC^{+}_{2m}$ always contain three distinct points and the orbits under $\PGL_{2}$ will have same cardinality $|\PGL_{2}|$ . 
One obtains
$$
|\cC^{+}_{n}/\PGL_{2}|=c^{+}_{2m}/(q^{3}-q)
=\sum_{k=1}^{m-1}q^{k-1}|\cM_{2m-2k+1}|=\sum_{k=1}^{m-1}q^{k-1}[m-k]_{q^{2}}.
$$
Assume $\Char(\F_{q})>2$ and $n=2m$ with odd $m$.
We have $\cC_{2}^{+}=\cC_{2}$ so that $c^{+}_{2}=c_{2}=q(q+1)$ and we obtain from the recursion
$$
c^{+}_{2m}=c_{2m-1}+qc_{2m-3}+q^{2}c_{2m-5}+\ldots+q^{m-2}c_{3}+q^{m-1}c_{2}.
$$
The space $\cC^{+}_{2m}$ does contain the elements of the form $(v_{1},v_{2}, v_{1},v_{2}, \ldots, v_{1},v_{2})$. These elements are all the same under the action of  $\PGL_{2}$.  The other elements of $\cC^{+}_{2m}$ always contain three distinct points and their orbits under $\PGL_{2}$ have same cardinality $q^{3}-q$. 
One obtains
\begin{eqnarray*}
|\cC^{+}_{n}/\PGL_{2}|&=&(c^{+}_{2m}-q(q+1))/(q^{3}-q)+1\\
&=&\sum_{k=1}^{m-1}q^{k-1}|\cM_{2m-2k+1}|+(q^{m-1}-1)q(q+1)/(q^{3}-q)+1\\
&=&\sum_{k=1}^{m-1}q^{k-1}[m-k]_{q^{2}}+[m-1]_{q}+1.
\end{eqnarray*}
Assume $\Char(\F_{q})=2$. In that case $c_{2}^{+}=c_{2}^{-}=c_{2}=q(q+1)$. The computations are the same as in the previous case, independently of $m$ is odd or even.
\end{proof}
Finally one gives alternative expressions in the formula of Lemme \ref{lemchat}.
One first notices that
\begin{eqnarray*}
q^{k-1}[m-k]_{q^{2}}+q^{k}[m-k-1]_{q^{2}}&=&q^{k-1}\frac{q^{2(m-k)}-1+q^{(2(m-k)-1)}-q}{q^{2}-1}\\[4pt]
&=&\frac{q^{2m-k-2}-q^{k-1}}{q-1}.
\end{eqnarray*}
In the sums one can group the terms for $k=2\ell-1$ and $k=2\ell$. If $m=2m'$ one obtains
\begin{equation*}
\begin{array}{rcl}\displaystyle
 \sum_{k=1}^{m-1}q^{k-1}[m-k]_{q^{2}} &=   &\displaystyle \frac{1}{q-1}\sum_{\ell=1}^{m'-1}(q^{2m-2\ell-1}-q^{2\ell-2})+q^{m-2}  \\[10pt]
  &=   & \displaystyle \frac{1}{q-1}(q^{m-1}[m']_{q^{2}}-[m']_{q^{2}}) \\[12pt]
  &  = &  [ m-1]_{q}[m']_{q^{2}}
\end{array}
\end{equation*}
and if $m=2m'+1$ one obtains
\begin{equation*}
\begin{array}{rcl}
\displaystyle \sum_{k=1}^{m-1}q^{k-1}[m-k]_{q^{2}} &=   & \displaystyle\frac{1}{q-1}\sum_{\ell=1}^{m'}(q^{2m-2\ell-1}-q^{2\ell-2})  \\[10pt]
  &=   &\displaystyle  \frac{1}{q-1}(q^{m}[m']_{q^{2}}-[m']_{q^{2}}) \\[12pt]
  &  = &  [ m]_{q}[m']_{q^{2}}
\end{array}
\end{equation*}
One observes that  in both cases $ \sum_{k=1}^{m-1}q^{k-1}[m-k]_{q^{2}}={ m \choose 2}_{q}$.
Theorem \ref{chat} is proved.
\end{proof}

\subsection{Counting tame friezes over $\F_{q}$}\label{friq}
A tame frieze of width $n-3$ is determined by $(a_{1}, \ldots, a_{n})$ satisfying \eqref{mateq}. It is clear that if $(a_{1}, \ldots, a_{n})$ satisfies \eqref{mateq} then so do $(a_{2}, \ldots, a_{n}, a_{1})$.
The corresponding friezes look the same up to an horizontal shift. We count friezes regardless symmetries, i.e. we count all the solutions of \eqref{mateq} considering that $(a_{1}, a_{2},\ldots, a_{n})$ and  $(a_{2}, \ldots, a_{n}, a_{1})$ give two different friezes whenever  $(a_{1}, a_{2},\ldots, a_{n})\not=(a_{2}, \ldots, a_{n}, a_{1})$.

It is  also clear that if $(a_{1}, \ldots, a_{n})$ satisfies \eqref{mateq} then so do $(a_{n}, a_{n-1}, \ldots, a_{1})$. The corresponding friezes look the same up to vertical reflection. They count for two different friezes whenever  $(a_{1}, a_{2},\ldots, a_{n})\not=(a_{n}, a_{n-1}, \ldots, a_{1})$. 

To count all the tame friezes, or equivalently the solutions of \eqref{mateq}, over  $\F_{q}$, we use the geometric realizations of the previous sections.\\

\noindent
\textbf{Proof of Theorem \ref{nbfriz}} (i) Let $n=2m+3$.
The construction in Section \ref{pftheoodd} gives a bijection between the set of tame friezes of width $n-3$ with odd $n$ and the set 
$\cM_{0,n}(\K)$. Theorem \ref{contmon} gives the number of elements in the sets when $\K=\F_{q}$.
Hence we get \eqref{f2m}.

(ii)
Let $n=2m$. 
By Theorem \ref{frizmodeven} 
every class in $\cC^{+}_{n}/\PGL_{2}$ leads to a $n$-tuple  $(a_{1}, \ldots, a_{n})$ and a family of friezes defined by the first rows $(\l a_{1}, \frac{1}{\l} a_{2},  \ldots, \l a_{n-1} , \frac{1}{\l}a_{n})$, for all $\l\not=0$. 

Assume $m$ even and $\Char(\F_{q})\not=2$. In this case 
one has $(a_{1}, \ldots, a_{n})\not=(0, \ldots, 0)$ because the elements in $\cC^{+}_{n}$ are not of the form $(v_{1},v_{2}, v_{1},v_{2}, \ldots, v_{1},v_{2})$. Therefore the corresponding family of friezes  contains exactly $(q-1)$ distinct friezes.
One thus obtains
$$
f_{2m-3}=(q-1)|\cM^{+}_{0,2m}|=(q-1){m \choose 2 }_{q},
$$
by Theorem \ref{chat}.

Assume $m$ odd and $\Char(\F_{q})\not=2$, or assume $\Char(\F_{q})=2$. In this case $\cC^{+}_{n}$
contains the elements of the form $(v_{1},v_{2}, v_{1},v_{2}, \ldots, v_{1},v_{2})$. They are all in the same class in $\cM^{+}_{0,2m}$ and this class gives the coefficients $(a_{1}, \ldots, a_{n})=(0, \ldots, 0)$ and leads to a unique frieze.
The other classes in $\cM^{+}_{0,2m}$ lead to families of friezes defined by $(\l a_{1}, \frac{1}{\l} a_{2},  \ldots, \l a_{n-1} , \frac{1}{\l}a_{n})$, $\l\not=0$ with $(a_{1}, \ldots, a_{n})\not=(0, \ldots, 0)$. These families contain $(q-1)$ distinct friezes.
Using Theorem \ref{chat} one thus obtains
\begin{eqnarray*}
f_{2m-3}&=&(q-1)\left(|\cM^{+}_{0,2m}|-1\right) +1
=(q-1)\left({m \choose 2 }_{q}+[m-1]_{q}\right) +1\\
&=&(q-1){m \choose 2 }_{q}+q^{m-1}.
\end{eqnarray*}
Theorem \ref{nbfriz} is proved.

\begin{rem}
The polynomial $(q-1){m \choose 2 }_{q}$ is a sum of odd powers of $q$ minus a sum of even powers of $q$. More precisely one has
$$
(q-1){m \choose 2 }_{q}=
\begin{cases}
 q^{2m-3}+q^{2m-5}+\ldots +q^{m+1}+q^{m-1}-q^{m-2}-q^{m-4}-\ldots-q^{2}-1,&\!\text{($m$ even)},\\[8pt]
q^{2m-3}+q^{2m-5}+\ldots +q^{m+2}+q^{m}-q^{m-3}-q^{m-5}-\ldots-q^{2}-1, &\!\text{($m$ odd)}.\\[1pt]
\end{cases}
$$

\end{rem}

\begin{rem}
It is known that friezes of width $n-3$ are related to the cluster algebra of type $\mA_{n-3}$ \cite{CaCh}. The results of Theorems \ref{contmon} and \ref{chat} agree with the results of \cite[Prop 3.2]{Cha} (see also \cite[Prop 4.10]{Cha2}).
\end{rem}
\subsection{Friezes of width 1 over finite fields}
Recall that the tameness condition \S\ref{tamsec} is defined in the frieze extended top and bottom with rows of 0's and $-1$'s. For simplicity we omit these rows in the arrays.

Let $\K=\F_{q}$
be a finite field of characteristic bigger than 2. Tame friezes of width 1 are exactly the following,$$
\begin{array}{cccccccccccccccccccccccc}
\cdots&&1&&1&& 1&&1&&1&&1&& \cdots\\
&\cdots&&2\a^{-1}&&\a&&2\a^{-1}&&\a&&2\a^{-1}&&\cdots&&\\
\cdots&&1&&1&& 1&&1&&1&&1&& \cdots\\
\end{array}
$$
 for all  $\a\not=0$.

Let $\K=\F_{q}$ 
be a finite field of characteristic 2. Tame friezes of width 1 are exactly the following
$$
\begin{array}{cccccccccccccccccccccccc}
\cdots&&1&&1&& 1&&1&&1&&1&& \cdots\\
&\cdots&&0&&\a&&0&&\a&&0&&\cdots&&\\
\cdots&&1&&1&& 1&&1&&1&&1&& \cdots\\
\end{array}
$$
with $\a\in\K$. If $\a\not=0$ the above array counts for two different friezes. 

Hence we have just established that
the number of tame friezes over $\F_{q}$ is $2q-1$ if the characteristic is 2 and $q-1$ otherwise, which agrees with the value of $f_{1}$ given in Theorem \ref{nbfriz}.

\subsection{Examples of friezes of width 2 over finite fields}
A frieze of width 2 is determined by 5 consecutive coefficients $(a_{1}, a_{2}, \ldots, a_{5})$ on the first row. Under a dihedral action, i.e. cyclic permutations $(a_{2}, a_{3}, \ldots, a_{1})$ and/or reversion $(a_{5}, a_{4}, \ldots, a_{1})$, a frieze can produce up to 10  friezes.

By Theorem \ref{nbfriz} there exist $1+q^{2}$ tame friezes of width 2 over $\F_{q}$.

For $q=2$ the five friezes contain the following fundamental domain
$$
\begin{array}{cccccccccccccccccccccccc}
1&& 1&&1&&1\\
&1&&1&&1\\
&&0&&0&\\
&&&1&
\end{array}
$$
and are defined by the coefficients $(1,1,1,0,0)$ on the first row, up to cyclic permutation.

For $q=4$, let $\F_{4}=\{0,1,\a, \b\}$ with $\b=1+\a=\a^{-1}$. The 17 friezes are the same as the five ones above for $\F_{2}$ together with the friezes containing one of the following three fundamental domains
$$
\begin{array}{cccccccccccccccccccccccc}
1&& 1&&1&&1\\
&\a&&0&&\b\\
&&1&&1&\\
&&&1&
\end{array}\qquad
\begin{array}{cccccccccccccccccccccccc}
1&& 1&&1&&1\\
&\a&&\a&&\a\\
&&\a&&\a&\\
&&&1&
\end{array}\qquad
\begin{array}{cccccccccccccccccccccccc}
1&& 1&&1&&1\\
&\b&&\b&&\b\\
&&\b&&\b&\\
&&&1&
\end{array}$$
Note that the domain on the left contributes to 10 friezes under dihedral action.

For $q=3$ the 10 friezes are obtained by cyclic permutations of $(2,1,0,1,2)$ and $(0,2,2,2,0)$.
For $q=5$ the 26 friezes are obtained by cyclic permutations and reversions of (2,1,3,1,2), (4,4,4,0,0,) (4,2,0,2,4) (1,4,4,3,0), (3,3,3,3,3).

\subsection{Friezes of width 3 over $\F_{2}$ and $\F_{3}$}
A frieze of width 3 is determined by 6 consecutive coefficients $(a_{1}, a_{2}, \ldots, a_{6})$ on the first row. Under a dihedral action, i.e. cyclic permutations $(a_{2}, a_{3}, \ldots, a_{1})$ and/or reversion $(a_{6}, a_{5}, \ldots, a_{1})$, a frieze can produce up to 12 friezes.

By Theorem \ref{nbfriz} there exist $q^{3}+q^{2}-1$ tame friezes of width 3 over $\F_{q}$.

For $q=2$, there are 11 tame friezes of width 3 given by the cycles $(1,0,1,0,0,0)$, $(1,1,0,1,1,0)$, $(1,1,1,1,1,1)$ and $(0,0,0,0,0,0)$. By dihedral action these cycles contribute respectively to 6, 3, 1 and 1 friezes.

For $q=3$, there are 35 tame friezes of width 3 given by the cycles $(0,0,0,2,0,1)$, $(1,1,0,2,2,0)$, $(1,0,2,1,0,2)$, $(1,0,1,0,1,0)$, $(2,0,2,0,2,0)$ and $(0,0,0,0,0,0)$.  By dihedral action these cycles contribute respectively to 12, 6, 6, 6, 2, 2  and 1 friezes.

\subsection{Combinatorial models}
Theorem \ref{nbfriz} gives the number of friezes over an arbitrary finite field but how to produce all the friezes (or all solutions of \eqref{mateq}) remains a natural open problem. 

When working over $\F_{p}$ with $p$ prime, the original construction of Conway-Coxeter \cite{CoCo} for friezes over $\Z_{\geq 0}$ by means of triangulations of polygons provide solutions to this problem by reducing modulo $p$. But not all the solutions.
A construction producing all the friezes over $\F_{2}$ and $\F_{3}$ has been given \cite{Mab}.

The construction of Conway-Coxeter \cite{CoCo} has been generalized to other cases \cite{Ovs3d}, \cite{Cun4}, \cite{CuHo}, \cite{HoJo3}.  These generalized models could inspire a model for friezes over finite  fields.\\

\noindent
\textbf{Open problem:} Find a combinatorial model to construct all friezes over $\F_{q}$.

\subsection{Numerology}
We compute the first terms of the sequence $(f_{n})$ given by Theorem \ref{nbfriz}:  
\begin{center}
$$
   \begin{array}{l|l|l}
&\Char(\F_{q})>2 & \Char(\F_{q})=2\\[4pt]
\hline
f_{1}&q-1 &2q-1\\[4pt]
f_{2}& q^{2}+1 &\text{idem}   \\[4pt]
f_{3}&q^{3}+q^{2}-1&\text{idem}  \\[4pt]
f_{4}&q^{4}+q^{2}+1&\text{idem}  \\[4pt]
f_{5}& q^{5}+q^{3}-q^{2}-1&q^{5}+2q^{3}-q^{2}-1\\[4pt]
f_{6}&q^{6}+q^{4}+q^{2}+1&\text{idem}  \\[4pt]
f_{7}&q^{7}+q^{5}+q^{4}-q^{2}-1 &\text{idem}   \\[4pt]
f_{8}&q^{8}+q^{6}+q^{4}+q^{2}+1 &\text{idem}     \\[4pt]
f_{9}&q^{9}+q^{7}+q^{5}-q^{4}-q^{2}-1 &q^{9}+q^{7}+2q^{5}-q^{4}-q^{2}-1 
   \end{array}
   $$
 \end{center}

The table below gives the numerical values for $q=2,3$ and $4$.

\begin{center}
$$
   \begin{array}{|l||c|c|c|c|c|c|c|}
     \hline
     q\setminus f_{n} & f_{1} & f_{2} & f_{3} & f_{4} & f_{5} & f_{6} & f_{7} \\[4pt] \hline
    q=2
    & 3 & 5 & 11 &21 &43 &85 &171\\[4pt]\hline
     q=3
    & 2 & 10 & 35 &91 &260 &820&2501\\[6pt]
     \hline
      q=4
    & 7 & 17 & 79 &273 &1135 &4369&17696\\[6pt]
     \hline

   \end{array}
   $$
 \end{center}
 The sequence for $q=2$ is the Jacobsthal sequence, see A001045 in \cite{OEIS}.  Indeed, from the explicit formulas of Theorem \ref{nbfriz}, one can check  that for $q=2$ one has
 $$
 f_{n}=f_{n-1}+2f_{n-2}
 $$
 for all $n\geq 3$.
 
The other sequences have no entry so far in OEIS.

%%%%%%%%%%%%%%%%%%%%%%
%%%%%%%%%%%%%%%%%%%%%%
\appendix
\section{Partitions of a circular set}\label{part}
%%%%%%%%%%%%%%%%%%%%%%
%%%%%%%%%%%%%%%%%%%%%%
As an application of the computation in Section \ref{count} and independently of Section~\ref{mon} and Section~\ref{friz} we solve an enumeration problem of restricted partitions of a circular set.

\subsection{Definition}
We consider the set $\{1, 2, \ldots, n\}$ as the labels of $n$ cyclically ordered points on the circle, i.e. 1 and $n$ are consecutive points.
We denote by $A_{k,n}$ the number of partitions of the set into $k$ parts such that no two consecutive points are in the same part.

\begin{ex} One always has $A_{1,n}=0$ and $A_{n,n}=1$. For $n=3$, one has $A_{1,3}=A_{2,3}=0$ and $A_{3,3}=1$.
For $n=4$ one has $A_{1,4}=0$, $A_{2,4}=1$, since the only allowed partition into two parts is exactly $\{1,3\}\sqcup\{2,4\}$ and $A_{3,4}=2$ since the allowed partitions into three parts are exactly $\{1,3\}\sqcup\{2\}\sqcup\{4\}$ and $\{2,4\}\sqcup\{1\}\sqcup\{3\}$.
For $n=5$, one has $A_{3,5}=5$, since the partitions of $\{1, \ldots ,5\}$ into 3 parts are $\{1,3\}\sqcup\{2,4\}\sqcup\{5\}$ together with the 5 cyclic permutations of the indices. 
\end{ex}

Our next goal is to deduce the exact value of $A_{k,n}$ from \S\ref{count1}. These values are known, see A261139 and A105794 of \cite{OEIS}. The corresponding generating function is given in \cite{KnLo}. They are also a particular case of graphical Stirling numbers \cite{GaTh}.

\subsection{Recounting points of $\cC_{n}(\F_q)$}
We consider the set of configurations
$$
\cC_{n}(\F_{q})=\{ (v_{1}, v_{2}, \ldots, v_{n})\in (\pP^{1}(\F_{q}))^{n}, \;v_{1}\not=v_{n}, \;v_{i}\not=v_{i+1}, \;\forall i\}
$$ 
introduced in~\S\ref{monq}.
The cardinal $c_{n}=|\cC_{n}(\F_q)|$ has been computed in Lemma \ref{lemcn}. It can be related to the numbers $A_{k,n}$.

\begin{prop} For $n\geq 2$, one has
$$c_{n}=q(q+1)\sum_{k=2}^{n}A_{k,n}\prod_{j=1}^{k-2}(q-j).$$
\end{prop}
\begin{proof}

Let us choose a partition $S_{1}\sqcup\ldots\sqcup S_{k}$ of cyclic $\{1, 2, \ldots, n\}$ such that  no two consecutive integers are in the same subset $S_{\ell}$. Consider a configuration in $\cC_{n}$ satisfying $v_{i}=v_{j}$ if and only if $i$ and $j$ belong to the same part $S_{\ell}$. There are exactly $(q+1)q\prod_{j=1}^{k-2}(q-j)$ such configurations, corresponding to a choice of $k$ points in $\pP^{1}(\F_{q})$. Hence the results.
\end{proof}

In the next section we use the exact value of $c_{n}$ given by \eqref{formcn} to deduce the value of $A_{k,n}$.

\subsection{Determining $A_{k,n}$}
The problem has now been translated into a simple problem of linear algebra consisting in expanding a given polynomial $P(q)\in \Q[q]$ into the basis of the falling factorials 
$$(q)_{k}:=q(q-1)\ldots (q-(k-1)).$$
This exercise has been probably already solved in the literature. We write down a brief solution.

Let $P$ be a polynomial of degree $n$ and denote by $B_{k}$ its coefficients in the basis $\{(q)_{k}\}_{k}$, i.e. 
$$P(q)=\sum_{k=0}^{n}B_{k}\,(q)_{k}.$$
\begin{lem} 
For $0\leq k\leq n$, 
 one has
$$
B_{k}=(-1)^{k}\sum_{j=0}^{k}\dfrac{(-1)^{j}}{j!(k-j)!}P(j).
$$
\end{lem}
\begin{proof}
By evaluating at $q=j$, and setting $p_{j}:=P(j)/j!$, one immediately obtains
$$
p_{j}=\sum_{k=0}^{j}\dfrac{1}{(j-k)!}B_{k},
$$
that can be written in a triangular linear system
\begin{equation*}
\begin{pmatrix}
p_{0}\\[5pt]
p_{1}\\[5pt]
p_{2}\\[5pt]
p_{3}\\[5pt]
\vdots\\[5pt]
p_{n}
\end{pmatrix}\,=
\begin{pmatrix}
\frac{1}{0!}\\[4pt]
\frac{1}{1!}&\frac{1}{0!}\\[4pt]
\frac{1}{2!}&\frac{1}{1!}&\frac{1}{0!}\\[4pt]
\frac{1}{3!}&
\frac{1}{2!}&\frac{1}{1!}&\frac{1}{0!}\\[4pt]
\vdots&\vdots&\vdots&&\ddots\\[6pt]
\frac{1}{n!}&\frac{1}{(n-1)!}&\frac{1}{(n-2)!}&\cdots& \frac{1}{1!}&\frac{1}{0!}\\[4pt]
\end{pmatrix}
\begin{pmatrix}
B_{0}\\[5pt]
B_{1}\\[5pt]
B_{2}\\[5pt]
B_{3}\\[5pt]
\vdots\\[5pt]
B_{n}
\end{pmatrix}
\end{equation*}
whose inverse system is
\begin{equation*}
\begin{pmatrix}
B_{0}\\[5pt]
B_{1}\\[5pt]
B_{2}\\[5pt]
B_{3}\\[5pt]
\vdots\\[5pt]
B_{n}
\end{pmatrix}\,=
\begin{pmatrix}
\frac{1}{0!}\\[4pt]
-\frac{1}{1!}&\frac{1}{0!}\\[4pt]
\frac{1}{2!}&-\frac{1}{1!}&\frac{1}{0!}\\[4pt]
-\frac{1}{3!}&
\frac{1}{2!}&-\frac{1}{1!}&\frac{1}{0!}\\[4pt]
\vdots&\vdots&&&\ddots\\[6pt]
\frac{(-1)^{n}}{n!}&\frac{(-1)^{n-1}}{(n-1)!}&\cdots&\cdots& -\frac{1}{1!}&\frac{1}{0!}\\[4pt]
\end{pmatrix}
\begin{pmatrix}
p_{0}\\[5pt]
p_{1}\\[5pt]
p_{2}\\[5pt]
p_{3}\\[5pt]
\vdots\\[5pt]
p_{n}
\end{pmatrix}.
\end{equation*}
Hence, one obtains
$$
B_{k}=(-1)^{k}\sum_{j=0}^{k}\dfrac{(-1)^{j}}{(k-j)!}p_{j}.
$$
\end{proof}

Applying the above lemma for $P=c_{n}/q(q+1)$ with $c_{n}$ given by \eqref{formcn} (note that this is indeed a polynomial in $q$ since $c_{n}$ vanishes for $q=0$ and $q=-1$), one immediately deduces the following formula.

\begin{prop}\label{Akn}
For $n\geq 2$ and $2\leq k \leq n$, one has
$$
A_{k,n}=(-1)^{k}\sum_{j=2}^{k}\dfrac{(-1)^{j}}{j!(k-j)!}\big(\,(j-1)^{n}+(-1)^{n}(j-1)\big).
$$
\end{prop}

\noindent
\textbf{Acknowledgements.}
I am deeply grateful to Valentin Ovsienko for stimulating discussions on the subject. This work is supported by the ANR project SC3A, ANR-15-CE40-0004-01.

\bibliographystyle{alpha}
\bibliography{BibFrisesCluster,BiblioMoy3}

\end{document}